\documentclass{article}

\usepackage{amsmath,amsfonts,amsthm,wasysym}
\usepackage{setspace}
\usepackage{tikz}
\usepackage{comment}

\newtheorem{thm}{Theorem}

\newtheorem{cnj}[thm]{Conjecture}
\newtheorem{cor}[thm]{Corollary}
\newtheorem{fct}[thm]{Fact}
\newtheorem{lem}[thm]{Lemma}

\def\a{{\alpha}}
\def\b{{\beta}}

\def\d{{\delta}}

\def\i{{\iota}}
\def\k{{\kappa}}

\def\p{{\pi}}

\def\s{{\sigma}}

\def\cP{{\cal P}}

\def\hp{{\hat{\p}}}

\def\dD{{\dot{D}}}
\def\dP{{\dot{P}}}

\def\cost{{\sf cost}}
\def\diam{{\sf diam}}
\def\dist{{\sf dist}}
\def\ecc{{\sf ecc}}
\def\pot{{\sf pot}}
\def\supp{{\sf supp}}

\def\rar{{\rightarrow}}

\definecolor{brwn}{RGB}{140, 70, 20}
\definecolor{gren}{RGB}{  0,140, 10}

\pgfmathsetmacro\Radi{10}
\pgfmathsetmacro\radi{5}

\begin{document}

\title{On the Target Pebbling Conjecture}

\author{
Glenn Hurlbert \thanks{
Department of Mathematics and Applied Mathematics,
Virginia Commonwealth University, 
1015 Floyd Ave, Richmond, VA 23220 USA
}
\footnote{ghurlbert@vcu.edu}
\and
Essak Seddiq \footnotemark[1]
\footnote{seddiqe@mymail.vcu.edu}
}
\maketitle
\date{}

\doublespacing

\begin{abstract}
Graph pebbling is a network optimization model for satisfying vertex demands with vertex supplies (called pebbles), with partial loss of pebbles in transit.
The pebbling number of a demand in a graph is the smallest number for which every placement of that many supply pebbles satisfies the demand.
The Target Conjecture (Herscovici-Hester-Hurlbert, 2009) posits that the largest pebbling number of a demand of fixed size $t$ occurs when the demand is entirely stacked on one vertex.
This truth of this conjecture could be useful for attacking many open problems in graph pebbling, including the famous conjecture of Graham (1989) involving graph products.
It has been verified for complete graphs, cycles, cubes, and trees.
In this paper we prove the conjecture for 2-paths and Kneser graphs over pairs.
\end{abstract}

\begin{quote}
\end{quote}
\begin{quote}
\end{quote}

\newpage

\section{Introduction}
Graph pebbling of the type we study here began as a method for proving a number-theoretic conjecture of Erd\H{o}s and Lemke (see \cite{Chung}), which was further applied to prove a group-theoretic conjecture of Kleitman and Lemke (see \cite{ElleHurl}), as well as to prove a result in $2$-adic analysis (see \cite{Knapp1,Knapp2}).
It has since grown into a network optimization model for satisfying vertex demands with vertex supplies (called pebbles), with partial loss of pebbles in transit.
While there are a number of different pebbling games on graphs, with a wide range of rules and applications (see \cite{GilLenTar,GureShel,HopPauVal,KiroPapa,Klawe,Liu,Parso,PateHewi,Sethi,StreTher}), this version is defined as follows.

For a finite, connected graph $G=(V,E)$, a {\it configuration} (or {\it supply}) $C$ is a non-negative integer-valued function on $V$, with {\it size} $|C|=\sum_{v\in V}C(v)$.
That is, $C(v)$ represents the number of pebbles on the vertex $v$, while $|C|$ denotes the total number of pebbles on $V$.
For a vertex $v$ such that $C(v)>0$, we define the configuration $C-v$ by $(C-v)(v)=C(v)-1$ and $(C-v)(u)=C(u)$ for all $u\not= v$.
Similarly, a {\it target distribution} (or {\it demand}) $D$ is a non-negative integer-valued function on $V$, with {\it size} $|D|=\sum_{v\in V}D(v)$.
That is, $D(v)$ equals the number of pebbles required to eventually place on the vertex $v$, while $|D|$ denotes the total demand of pebbles on $V$.
The notation $D-v$ is defined similarly: $(D-v)(v)=D(v)-1$ and $(D_v)(u)=D(u)$ for all $u\not=v$.
For a target $D$, define $\dD$ to be the multiset $\{v^{D(v)}\}_{v\in V}$.
($D(v)$ is the {\it multiplicity} of $v$ in $\dD$.)

For a configuration $C$, a {\it pebbling step} from $u$ to an adjacent vertex $v$ removes two pebbles from $u$ and places one of those pebbles on $v$; the resulting configuration $C'$ is defined by $C'(u)=C(u)-2$, $C'(v)=C(v)+1$, and $C'(w)=C(w)$ otherwise.
We say that $C$ is $D$-{\it solvable} (or that $C$ {\it solves} $D$) if $C$ can be converted via pebbling steps to a configuration $C^*$ such that $C^*(v)\ge D(v)$ for all $v\in V$; $C$ is $D$-{\it unsolvable} otherwise.
A sequence of such pebbling steps is called a $D$-{\it solution} ($r$-{\it solution} in the case that $\dD=\{r\}$).
Suppose that $r$ has $D(r)>0$ and that $\s$ is an $r$-solution given by the sequence of pebbling steps $\s_1, \ldots, \s_k$.
For each $1\le i\le k$, denote by $C^{(i)}$ the configuration resulting from $C^{(i-1)}$ via the pebbling step $\s_i$ (where $C^{(0)}=C$).
Then we define the configuration $C-\s=C^{(k)}-r$.
That is, $C-\s$ is the configuration resulting from removing all pebbles involved in the $r$-solution $\s$.
Thus, if $C-\s$ solves $D-r$, then $C$ solves $D$.

The {\it pebbling number of a demand $D$ in a graph $G$} is denoted $\p(G,D)$ and defined to be the smallest $m$ such that every configuration of size $m$ is $D$-solvable.
In the case that $|D|=1=D(r)$, we simply write $\p(G,r)$.
When $|D|=t=D(r)$ we say that $D$ is {\it stacked} (on $r$); in this case we may write that $C$ $t$-{\it fold solves} $r$ instead of that $C$ solves $D$, and use the notation $\p_t(G,r)=\p(G,D)$.
The $t$-{\it fold pebbling number of} $G$ is defined as $\p_t(G)=\max_{r\in V}\p_t(G,r)$; if $t=1$ we omit the subscript and avoid writing ``1-fold''.
As with many fractional analogues of graph theoretical invariants (chromatic number, clique number, matching number, etc. --- see \cite{ScheUllm}), the {\it fractional pebbling number} is defined to be $\hp(G)=\liminf_{n\rar\infty}\p_t(G)/t$.
It was proved in \cite{HerHesHurTpebb,HoMaOkZu} that $\hp(G)=2^{\diam(G)}$ for every graph $G$.

The original application of graph pebbling only involved the case $t=1$.
However, in \cite{Chung} the problem was immediately generalized so that the parameter $t$ could be used in an inductive manner to prove results on trees.
Similarly, the problem was further expanded to more general $D$\footnote{We note that the $D$-pebbling number was first introduced in \cite{CCFHPST} for the case $D(v)\ge 1$ for all $v\in V(G)$, and was called the {\it cover pebbling number}.
In that paper, the authors prove for trees that the largest $D$-unsolvable configuration is obtained by stacking the entire configuration on a single vertex --- the result for all graphs was proved in \cite{Sjos}.
It is then easy to calculate the size of such a configuration for each vertex, discern which vertex has the largest stack, and therefore derive the cover pebbling number.} in \cite{HerHesHurTpebb}, wherein is found the following Target Conjecture.

\begin{cnj}
\label{c:tTarget}
{\bf (Target Conjecture)}
\cite{HerHesHurTpebb}
Every graph $G$ satisfies $\p(G,D)\le \p_{|D|}(G)$ for every target distribution $D$.
\end{cnj}

The authors of \cite{HerHesHurTpebb} verified this conjecture for trees, cycles, complete graphs, and cubes.
The generalization to $D$ has proved useful in obtaining results inductively on powers of paths (see \cite{AlcoHurl}).
The hope is that it may be a powerful tool more generally, for example on chordal graphs, for which it has been conjectured that the pebbling numbers of chordal graphs of a certain type can be calculated in polynomial time (see \cite{AlcGutHur}).
Furthermore, one might suspect that the use of general targets could be helpful in attacking the famous conjecture of Graham (see \cite{Chung}) that $\p(G\Box H)\le\p(G)\p(H)$, where $\Box$ denotes the cartesian product of graphs.
Herscovici, et al. \cite{HerHesHurGraham}, generalize this to conjecture that $\p(G_1\Box G_2,D_1\times D_2)\le\p(G,D_1)\p(H,D_2)$.
The truth of Conjecture \ref{c:tTarget} may prove to be a useful tool in this direction.

In this paper we verify Conjecture \ref{c:tTarget} in Theorem \ref{t:2PathTarget} for the family of 2-paths, defined in Section \ref{s:2Paths}, and in Theorem \ref{t:Kneser} for the family of Kneser graphs $K(m,2)$, defined in Section \ref{s:Kneser}.

\subsection{Preliminaries}
\label{s:Prelim}

Before beginning, we introduce a few key concepts used in the proofs.
For a fixed vertex $r$ denote by $V_i(r)$ the set of all vertices at distance $i$ from $r$.
For a configuration $C$ we define a vertex $v$ to be a {\it zero} of $C$ if $C(v)=0$, and denote the number of zeros of $C$ by $z(C)$.
In addition, we define the {\it support} of $C$ ($\supp(C)$) to be the set of vertices $v$ with $C(v)>0$, and denote $s(C)=|\supp(C)|$.
Note that $|V| = s(C) + z(C)$ since every vertex either has at least one pebble or none at all.
Furthermore, we define the {\it potential} of $C$: $\pot(C)=\sum_{v\in V}\lfloor C(v)/2\rfloor$.
This equals the number of pairwise disjoint pairs of pebbles with pebbles of the same pair sitting on the same vertex (each such pair is called {\it a potential}); in other words, it is the total number of initial pebbling steps that can be made from $C$ (with pebbles on the same vertex being indistinguishable).

\begin{lem}
\label{l:potlem}
{\bf (Potential Lemma)}
Let $C$ be a configuration on a graph $G$ with potential $\pot(C)$.
Then the following hold.
\begin{enumerate}
    \item 
    \label{potlem1}
    $\pot(C)\ge\left\lceil\frac{|C|-|V|+z(C)}{2}\right\rceil$.
    \item
    \label{potlem2}
    If $C$ solves a distribution $D$ and $\supp(C)\cap \supp(D)=\emptyset$ then $\pot(C)\ge |D|$.
\end{enumerate}
\end{lem}

\begin{proof}
Part (\ref{potlem1}) is a simple consequence of the relation $|C|\le 2\pot(C)+s(C)$.
Part (\ref{potlem2}) holds since placing a pebble on a target requires a potential to make a pebbling step.
\end{proof}

Finally, for a configuration $C$, a path $v_1,\ldots,v_k$ is called a ${\it slide}$ if $C(v_1)\ge 2$ and $C(v_i)\ge 1$ for all $1<i<k$.
In addition, for a single target $r$, we define the {\it cost}, $\cost(\s)$, of an $r$-solution $\s$ to be the number of pebbles discarded by $\s$, including the pebble placed on $r$; thus it equals one more than the number of pebbling steps of $\s$.
For example, for $r=v_k$ in the slide above, the cost of the solution that moves a pebble from $v_1$ to $r$ equals $k$.
Consequently, the resulting configuration $C'$ of pebbles remaining to use to solve other targets has size $|C'|=|C|-\cost(\s)$.
We define an $r$-solution to be {\it cheap} if its cost is at most $2^{\ecc(r)}$; a configuration is $r$-{\it cheap} if it has a cheap $r$-solution.
A {\it cost}-$k$ solution is simply a solution of cost exactly $k$.
We say that a graph $G$ is $r$-{\it (semi)greedy} if every configuration of size at least $\pi(G,r)$
has a (semi)greedy $r$-solution; that is, every pebbling step in the solution decreases (does not increase) the distance of the moved pebble to $r$.
It is known, for example (see \cite{AlcoHurl}), that trees are greedy and that chordal graphs are semi-greedy.

\begin{lem}
\label{l:CheapLemma}
{\bf (Cheap Lemma)}
\cite{AGHS2T}
Given the graph $G$ with target $r$ let $H$ be an $r$-greedy spanning subgraph of $G$ preserving distances to $r$.
Then any configuration of $G$ of size at least $\pi(H,r)$ is $r$-cheap.
\end{lem}

In particular, a breadth-first-search spanning tree can play the role of $H$ in the Cheap Lemma.
In the case of 2-paths, below, we use a caterpillar as the best choice of such a tree.
A {\it caterpillar} is a tree that contains a path $P$ such that every vertex not on $P$ is adjacent to some vertex on $P$.

\section{2-Paths}
\label{s:2Paths}

\subsection{Definition and notation}

A {\it simplicial} vertex is a vertex whose set of neighbors forms a complete graph. 
A {\it chordal} graph is a graph with no induced cycle of length 4 or more.
A $k$-{\it path} is either a complete graph $K_k$ or $K_{k+1}$, or a graph $G$ with exactly two simplicial vertices $u$ and $v$ such that the neighborhood of $v$ is $K_k$ and $G-v$ is a $k$-path.

Let $G$ be a chordal graph with two simplicial vertices, which we denote $r$ and $s$, with the shortest path $P=(x_0,x_1,\ldots,x_d)$ between them, where $d = \dist(r,s) = \diam(G)$ ($r = x_0, s = x_d$). 
We call this path the {\it spine} of $G$.
We define a {\it fan} $F$ to be a subgraph of $G$ which consists of a path $Q = a, v_1, \ldots, v_k, c$, and an additional vertex $b$ which is adjacent to every vertex of $Q$ and where $a, b, c$ form a subpath of $P$ (e.g. $a=x_{i-1}$, $b=x_i$, and $c=x_{i+1}$ for some $0<i<d$).
We say the $F$ is an $ac$-{\it fan} with fan vertices $F' = \{v_1, \ldots, v_k\}$.
Then $G$ is an {\it overlapping fan graph} if, for every $0< i < d$, there is an $x_{i-1}x_{i+1}$-fan $F_i$ in $G$, every vertex of $G$ is in some fan of $G$, and $|F'_i \cap F'_{i+1}| \le 1$.
An example of a fan graph is seen in left diagram of Figure \ref{fig:fan}, where $d=4$, $k_1 = 4$, $k_2 = 3$ and $k_3 = 3$.

By Lemma 2.1 of \cite{AGHP2P}, every 2-path is an overlapping fan graph (and vice-versa).
In \cite{AGHP2P} we find the following result.
\begin{thm}
\label{t:2Path}
If $G$ is a 2-path with $\diam(G)=d$ then $\p_t(G)=t2^d+n-2d$.
\end{thm}

The point of this section is to prove this bound for all targets $D$ of size $t$, rather than just $t$ targets on the same vertex --- see Theorem \ref{t:2PathTarget}.

\begin{figure}
\begin{center}
\begin{tikzpicture}[scale=1.4]
\tikzstyle{every node}=[draw,circle,fill=black,inner sep=2pt,label distance=0.02cm]
\draw node (r) [label=left: {$r$}] at (0,0) {};
\draw node (x1) [label=225: {$x_1$}] at (1,0) {};
\draw node (x2) [label=above: {$x_2$}] at (2,0) {};
\draw node (x3) [label=315: {$x_3$}] at (3,0) {};
\draw node (s) [label=right: {$s$}] at (4,0) {};
\draw[line width=1.0pt] (r) -- (x1) -- (x2) -- (x3) -- (s);
\def \j {1}
\def \a {3}
\def \A {180/(1+\a)}
\foreach \i in {1,...,\a}
    \def \m {180-\i*\A}
    \draw node [label=\m: {$v_{\j,\i}$}] at ({\j+cos(\m)},{sin(\m)}) {};
\foreach \i in {0,...,\a}
    \def \m {180-\i*\A}
    \def \mm {180-(\i+1)*\A}
    \draw[line width=1.0pt] ({\j+cos(\m)},{sin(\m)}) -- ({\j+cos(\mm)},{sin(\mm)});
\foreach \i in {1,...,\a}
    \def \m {180-\i*\A}
    \draw[line width=1.0pt] ({\j+cos(\m)},{sin(\m)}) -- (\j,0);
\def \j {2}
\def \a {4}
\def \A {180/(1+\a)}
\foreach \i in {1,...,\a}
    \def \m {180+\i*\A}
    \draw node [label=\m: {$v_{\j,\i}$}] at ({\j+cos(\m)},{sin(\m)}) {};
\foreach \i in {0,...,\a}
    \def \m {180+\i*\A}
    \def \mm {180+(\i+1)*\A}
    \draw[line width=1.0pt] ({\j+cos(\m)},{sin(\m)}) -- ({\j+cos(\mm)},{sin(\mm)});
\foreach \i in {1,...,\a}
    \def \m {180+\i*\A}
    \draw[line width=1.0pt] ({\j+cos(\m)},{sin(\m)}) -- (\j,0);
\def \j {3}
\def \a {2}
\def \A {180/(1+\a)}
\foreach \i in {1,...,\a}
    \def \m {180-\i*\A}
    \draw node [label=above: {$v_{\j,\i}$}] at ({\j+cos(\m)},{sin(\m)}) {};
\foreach \i in {0,...,\a}
    \def \m {180-\i*\A}
    \def \mm {180-(\i+1)*\A}
    \draw[line width=1.0pt] ({\j+cos(\m)},{sin(\m)}) -- ({\j+cos(\mm)},{sin(\mm)});
\foreach \i in {1,...,\a}
    \def \m {180-\i*\A}
    \draw[line width=1.0pt] ({\j+cos(\m)},{sin(\m)}) -- (\j,0);
\end{tikzpicture}
$\qquad$
\begin{tikzpicture}[scale=1.4]
\tikzstyle{every node}=[draw,circle,fill=black,minimum size=1pt,inner sep=2pt]
\draw node (r) [label=left: {$r$}] at (0,0) {};
\draw node (x1) [label=225: {$x_1$}] at (1,0) {};
\draw node (x2) [label=above: {$x_2$}] at (2,0) {};
\draw node (x3) [label=315: {$x_3$}] at (3,0) {};
\draw node (s) [label=right: {$s$}] at (4,0) {};
\draw[line width=1.0pt] (r) -- (x1) -- (x2) -- (x3) -- (s);
\def \j {1}
\def \a {3}
\def \A {180/(1+\a)}
\foreach \i in {1,...,\a}
    \def \m {180-\i*\A}
    \draw node [label=\m: {$v_{\j,\i}$}] at ({\j+cos(\m)},{sin(\m)}) {};
\foreach \i in {0,...,0}
    \def \m {180-\i*\A}
    \def \mm {180-(\i+1)*\A}
    \draw[line width=1.0pt] ({\j+cos(\m)},{sin(\m)}) -- ({\j+cos(\mm)},{sin(\mm)});
\foreach \i in {2,...,\a}
    \def \m {180-\i*\A}
    \draw[line width=1.0pt] ({\j+cos(\m)},{sin(\m)}) -- (\j,0);
\def \j {2}
\def \a {4}
\def \A {180/(1+\a)}
\foreach \i in {1,...,\a}
    \def \m {180+\i*\A}
    \draw node [label=\m: {$v_{\j,\i}$}] at ({\j+cos(\m)},{sin(\m)}) {};
\foreach \i in {0,...,0}
    \def \m {180+\i*\A}
    \def \mm {180+(\i+1)*\A}
    \draw[line width=1.0pt] ({\j+cos(\m)},{sin(\m)}) -- ({\j+cos(\mm)},{sin(\mm)});
\foreach \i in {2,...,\a}
    \def \m {180+\i*\A}
    \draw[line width=1.0pt] ({\j+cos(\m)},{sin(\m)}) -- (\j,0);
\def \j {3}
\def \a {2}
\def \A {180/(1+\a)}
\foreach \i in {1,...,\a}
    \def \m {180-\i*\A}
    \draw node [label=above: {$v_{\j,\i}$}] at ({\j+cos(\m)},{sin(\m)}) {};
\foreach \i in {0,...,0}
    \def \m {180-\i*\A}
    \def \mm {180-(\i+1)*\A}
    \draw[line width=1.0pt] ({\j+cos(\m)},{sin(\m)}) -- ({\j+cos(\mm)},{sin(\mm)});
\foreach \i in {2,...,\a}
    \def \m {180-\i*\A}
    \draw[line width=1.0pt] ({\j+cos(\m)},{sin(\m)}) -- (\j,0);
\end{tikzpicture}
\caption{An overlapping fan graph of diameter 4 on the left. The BFS tree rooted at $r$ on the right.}
\label{fig:fan}
\end{center}
\end{figure}
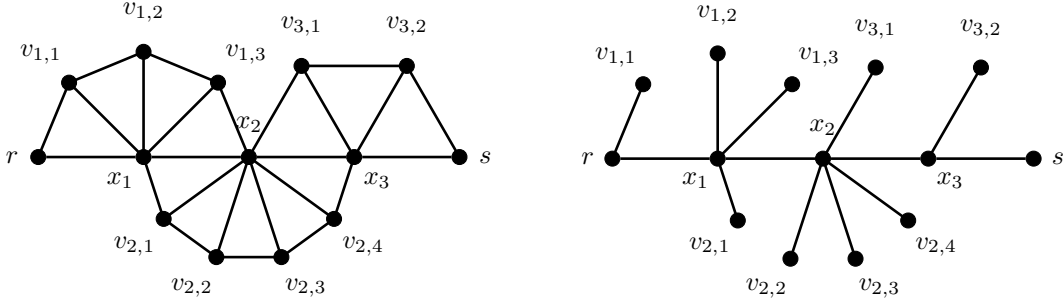

\subsection{Preliminary facts and lemmas}
We present some important preliminary facts and lemmas necessary for the proof of Theorem \ref{t:2PathTarget}.
These first two facts are used to determine the pebbling numbers of specially constructed trees which are essential to the proof.
Suppose $T$ is a tree with a root $r$. 
A {\it path partition} $\cP$ of $T$ is a set of pairwise edge-disjoint directed paths whose union is $T$. 
It is an $r$-{\it path partition} if $r$ is an endpoint of the longest path of $\cP$. 
A path partition is said to {\it majorize} ($\succ$) another if its non-increasing sequence of the path sizes majorizes that of the other --- that is, $(a_1, a_2, \ldots, a_i) \succ (b_1, b_2,\ldots, b_t)$ if and only if $a_i > b_i$, where $i = \min\{j:a_j \not = b_j\}$.
An $r$-path partition of a tree $T$ is said to be {\it maximum} if it majorizes all other $r$-path partitions.

\begin{fct}
\label{f:tree}
\cite{Chung}
Let $T$ be a tree rooted at $r$, and let $\cP =\{P_1,...,P_k\}$ be a maximum $r$-path partition of $T$.
Denote by $a_i$ the length of the path $P_i$, and suppose that the paths are labeled so that $a_i \ge a_{i+1}$ for all $1\le i<k$.
(Note that $a_1 = \ecc(r)$.)
Then $\p_t(T,r) = t2^{a_1} + \sum_{i=2}^{k} 2^{a_i} - k + 1$ for all $t\ge 1$.
\end{fct}

Let $T$ be any Breadth First Search (BFS) spanning tree of a 2-path $G$, rooted at $r$.
Then $T$ preserves all distances from $r$;
that is, $\dist_T(v,r)=\dist_G(v,r)$ for all vertices $v$.
For a simplicial vertex $r$, we define $T_r$ (the {\it spinal} tree of $r$) to be such a BFS tree chosen in a specific manner, namely, with the priority of choosing vertices of the spine $P_1$ (from the maximum $r$-path partition) before other vertices whenever possible.
As we see from Figure \ref{fig:fan}, this produces a caterpillar, which has a very simple pebbling number formula, according to Fact \ref{f:tree}, because its path partition has only one long path.
Let $S$ be the set of two simplicial vertices of $G$.

\begin{figure}
\begin{center}
\begin{tikzpicture}[scale=1.0]
\tikzstyle{every node}=[draw,circle,fill=black,inner sep=2pt]
\draw node (r) at (0,0) {};
\draw node (x1) [label=270: {$x_{i-1}$}] at (1,0) {};
\draw node (x2) [label=225: {$x_i$}] at (2,0) {};
\draw node (x3) at (3,0) {};
\draw node (s) at (4,0) {};
\draw[line width=1.0pt] (-.75,0) -- (r) -- (x1) -- (x2) -- (x3) -- (s) -- (4.75,0);
\draw node[inner sep=.7pt] at (-1.0,0) {};
\draw node[inner sep=.7pt] at (-1.2,0) {};
\draw node[inner sep=.7pt] at (-1.4,0) {};
\draw node[inner sep=.7pt] at (5.0,0) {};
\draw node[inner sep=.7pt] at (5.2,0) {};
\draw node[inner sep=.7pt] at (5.4,0) {};
\def \j {1}
\def \A {180/4}
\def \B {11*180/16}
    \draw node (u) at ({1+cos(\B)},{sin(\A)}) {};
    \draw[line width=1.0pt] (r) -- (u) -- ({2+cos(180-\A)},{sin(180-\A)});
    \draw[line width=1.0pt] (x1) -- (u);
\def \j {2}
\def \a {3}
\def \A {180/(1+\a)}
    \draw node (v1) at ({\j+cos(180-1*\A)},{sin(180-1*\A)}) {};
    \draw node (v2) at ({\j+cos(180-2*\A)},{sin(180-2*\A)}) {};
    \draw node (v3) at ({\j+cos(180-3*\A)},{sin(180-3*\A)}) {};
    \draw[line width=1.0pt] (x1) -- (v1) -- (v2) -- (v3) -- (x3);
    \draw[line width=1.0pt] ({\j+cos(180-1*\A)},{sin(180-1*\A)}) -- (\j,0);
    \draw[line width=1.0pt] ({\j+cos(180-2*\A)},{sin(180-2*\A)}) -- (\j,0);
    \draw[line width=1.0pt] ({\j+cos(180-3*\A)},{sin(180-3*\A)}) -- (\j,0);
\def \j {3}
\def \a {3}
\def \A {180/(1+\a)}
    \draw node (w1) at ({\j+cos(180+1*\A)},{sin(180+1*\A)}) {};
    \draw node (w2) at ({\j+cos(180+2*\A)},{sin(180+2*\A)}) {};
    \draw node (w3) at ({\j+cos(180+3*\A)},{sin(180+3*\A)}) {};
    \draw[line width=1.0pt] (x2) -- (w1) -- (w2) -- (w3) -- (s);
    \draw[line width=1.0pt] ({\j+cos(180+1*\A)},{sin(180+1*\A)}) -- (\j,0);
    \draw[line width=1.0pt] ({\j+cos(180+2*\A)},{sin(180+2*\A)}) -- (\j,0);
    \draw[line width=1.0pt] ({\j+cos(180+3*\A)},{sin(180+3*\A)}) -- (\j,0);
\end{tikzpicture}
$\qquad$
\begin{tikzpicture}[scale=1.0]
\tikzstyle{every node}=[draw,circle,fill=black,inner sep=2pt]
\draw node (r) at (0,0) {};
\draw node (x1) [label=270: {$x_{i-1}$}] at (1,0) {};
\draw node (x2) [label=225: {$r$}] at (2,0) {};
\draw node (x3) at (3,0) {};
\draw node (s) at (4,0) {};
\draw[line width=1.0pt] (-.75,0) -- (r) -- (x1) -- (x2) -- (x3) -- (s) -- (4.75,0);
\draw node[inner sep=.7pt] at (-1.0,0) {};
\draw node[inner sep=.7pt] at (-1.2,0) {};
\draw node[inner sep=.7pt] at (-1.4,0) {};
\draw node[inner sep=.7pt] at (5.0,0) {};
\draw node[inner sep=.7pt] at (5.2,0) {};
\draw node[inner sep=.7pt] at (5.4,0) {};
\def \j {1}
\def \A {180/4}
\def \B {11*180/16}
    \draw node (u) at ({1+cos(\B)},{sin(\A)}) {};
    \draw[line width=1.0pt] (x1) -- (u);
\def \j {2}
\def \a {3}
\def \A {180/(1+\a)}
    \draw node (v1) at ({\j+cos(180-1*\A)},{sin(180-1*\A)}) {};
    \draw node (v2) at ({\j+cos(180-2*\A)},{sin(180-2*\A)}) {};
    \draw node (v3) at ({\j+cos(180-3*\A)},{sin(180-3*\A)}) {};
    \draw[line width=1.0pt] ({\j+cos(180-1*\A)},{sin(180-1*\A)}) -- (\j,0);
    \draw[line width=1.0pt] ({\j+cos(180-2*\A)},{sin(180-2*\A)}) -- (\j,0);
    \draw[line width=1.0pt] ({\j+cos(180-3*\A)},{sin(180-3*\A)}) -- (\j,0);
\def \j {3}
\def \a {3}
\def \A {180/(1+\a)}
    \draw node (w1) at ({\j+cos(180+1*\A)},{sin(180+1*\A)}) {};
    \draw node (w2) at ({\j+cos(180+2*\A)},{sin(180+2*\A)}) {};
    \draw node (w3) at ({\j+cos(180+3*\A)},{sin(180+3*\A)}) {};
    \draw[line width=1.0pt] (x2) -- (w1);
    \draw[line width=1.0pt] ({\j+cos(180+2*\A)},{sin(180+2*\A)}) -- (\j,0);
    \draw[line width=1.0pt] ({\j+cos(180+3*\A)},{sin(180+3*\A)}) -- (\j,0);
\end{tikzpicture}\\
\bigskip
\begin{tikzpicture}[scale=1.0]
\tikzstyle{every node}=[draw,circle,fill=black,inner sep=2pt]
\draw node (r) at (0,0) {};
\draw node (x1) [label=270: {$x_{i-1}$}] at (1,0) {};
\draw node (x2) [label=225: {$x_i$}] at (2,0) {};
\draw node (x3) at (3,0) {};
\draw node (s) at (4,0) {};
\draw[line width=1.0pt] (-.75,0) -- (r) -- (x1);
\draw[line width=1.0pt] (x2) -- (x3) -- (s) -- (4.75,0);
\draw node[inner sep=.7pt] at (-1.0,0) {};
\draw node[inner sep=.7pt] at (-1.2,0) {};
\draw node[inner sep=.7pt] at (-1.4,0) {};
\draw node[inner sep=.7pt] at (5.0,0) {};
\draw node[inner sep=.7pt] at (5.2,0) {};
\draw node[inner sep=.7pt] at (5.4,0) {};
\def \j {1}
\def \A {180/4}
\def \B {11*180/16}
    \draw node (u) at ({1+cos(\B)},{sin(\A)}) {};
    \draw[line width=1.0pt] (u) -- ({2+cos(180-\A)},{sin(180-\A)});
\def \j {2}
\def \a {3}
\def \A {180/(1+\a)}
    \draw node (v1) [label=above: {$r$}] at ({\j+cos(180-1*\A)},{sin(180-1*\A)}) {};
    \draw node (v2) at ({\j+cos(180-2*\A)},{sin(180-2*\A)}) {};
    \draw node (v3) at ({\j+cos(180-3*\A)},{sin(180-3*\A)}) {};
    \draw[line width=1.0pt] (x1) -- (v1) -- (v2);
    \draw[line width=1.0pt] ({\j+cos(180-1*\A)},{sin(180-1*\A)}) -- (\j,0);
    \draw[line width=1.0pt] ({\j+cos(180-3*\A)},{sin(180-3*\A)}) -- (\j,0);
\def \j {3}
\def \a {3}
\def \A {180/(1+\a)}
    \draw node (w1) at ({\j+cos(180+1*\A)},{sin(180+1*\A)}) {};
    \draw node (w2) at ({\j+cos(180+2*\A)},{sin(180+2*\A)}) {};
    \draw node (w3) at ({\j+cos(180+3*\A)},{sin(180+3*\A)}) {};
    \draw[line width=1.0pt] (x2) -- (w1);
    \draw[line width=1.0pt] ({\j+cos(180+2*\A)},{sin(180+2*\A)}) -- (\j,0);
    \draw[line width=1.0pt] ({\j+cos(180+3*\A)},{sin(180+3*\A)}) -- (\j,0);
\end{tikzpicture}
$\qquad$
\begin{tikzpicture}[scale=1.0]
\tikzstyle{every node}=[draw,circle,fill=black,inner sep=2pt]
\draw node (r) at (0,0) {};
\draw node (x1) [label=270: {$x_{i-1}$}] at (1,0) {};
\draw node (x2) [label=225: {$x_i$}] at (2,0) {};
\draw node (x3) at (3,0) {};
\draw node (s) at (4,0) {};
\draw[line width=1.0pt] (-.75,0) -- (r) -- (x1) -- (x2);
\draw[line width=1.0pt] (x3) -- (s) -- (4.75,0);
\draw node[inner sep=.7pt] at (-1.0,0) {};
\draw node[inner sep=.7pt] at (-1.2,0) {};
\draw node[inner sep=.7pt] at (-1.4,0) {};
\draw node[inner sep=.7pt] at (5.0,0) {};
\draw node[inner sep=.7pt] at (5.2,0) {};
\draw node[inner sep=.7pt] at (5.4,0) {};
\def \j {1}
\def \A {180/4}
\def \B {11*180/16}
    \draw node (u) at ({1+cos(\B)},{sin(\A)}) {};
    \draw[line width=1.0pt] (x1) -- (u);
\def \j {2}
\def \a {3}
\def \A {180/(1+\a)}
    \draw node (v1) at ({\j+cos(180-1*\A)},{sin(180-1*\A)}) {};
    \draw node (v2) at ({\j+cos(180-2*\A)},{sin(180-2*\A)}) {};
    \draw node (v3) [label=right: {$r$}] at ({\j+cos(180-3*\A)},{sin(180-3*\A)}) {};
    \draw[line width=1.0pt] (v2) -- (v3) -- (x3);
    \draw[line width=1.0pt] ({\j+cos(180-1*\A)},{sin(180-1*\A)}) -- (\j,0);
    \draw[line width=1.0pt] ({\j+cos(180-3*\A)},{sin(180-3*\A)}) -- (\j,0);
\def \j {3}
\def \a {3}
\def \A {180/(1+\a)}
    \draw node (w1) at ({\j+cos(180+1*\A)},{sin(180+1*\A)}) {};
    \draw node (w2) at ({\j+cos(180+2*\A)},{sin(180+2*\A)}) {};
    \draw node (w3) at ({\j+cos(180+3*\A)},{sin(180+3*\A)}) {};
    \draw[line width=1.0pt] ({\j+cos(180+1*\A)},{sin(180+1*\A)}) -- (\j,0);
    \draw[line width=1.0pt] ({\j+cos(180+2*\A)},{sin(180+2*\A)}) -- (\j,0);
    \draw[line width=1.0pt] ({\j+cos(180+3*\A)},{sin(180+3*\A)}) -- (\j,0);
\end{tikzpicture}
\caption{A 2-path with rooted trees $T_r$ for the three cases of internal roots $r$.}
\label{fig:examples}
\end{center}
\end{figure}
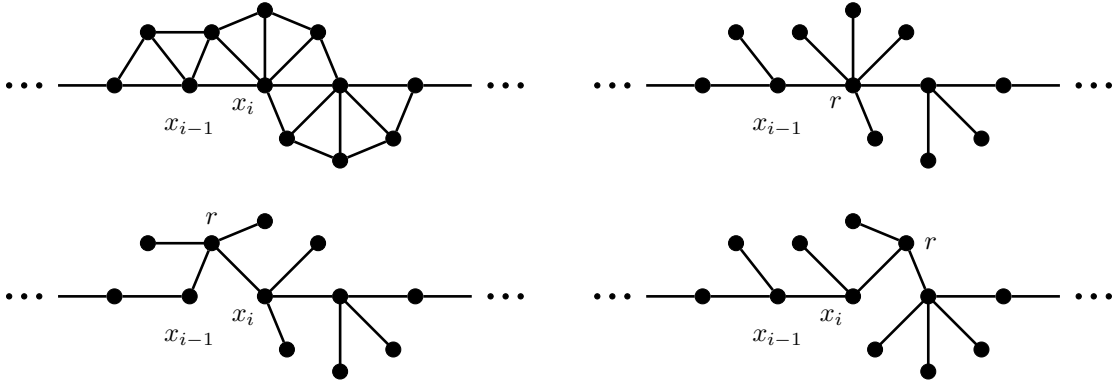

For an internal, spinal root $r$, let $F_i$ be the fan centered on $r$, and let $A_r$ be the set of vertices of $F_i$ that are in no other fan.
For each $v\in A_r$ we have that the edge $rv$ is a 2-path, which we denote by $G_r^v$.
Then $G-A_r$ is the union of two 2-paths, which we denote by $G_r^s$ for $s\in S$, that share only the vertex $r$.
Note that, for each $u\in S\cup A_r$, we have that $r$ and $u$ are the simplicial vertices of $G_r^u$; we denote a spinal tree of $r$ in $G_r^u$ by $T_r(G_r^u)$.
Now define $T_r=\cup_{u\in S\cup A_r}T_r(G_r^u)$.

For an internal, non-spinal root $r$, with $r$ in two fans, let $F_i$ and $F_{i+1}$ be the two fans that share $r$, centered on $x_i$ and $x_{i+1}$ respectively.
Then $G-x_ix_{i+1}$ is the union of two 2-paths $G_r^s$, for $s\in S$, that share only the vertex $r$.
Note that, for each $s\in S$, we have that $r$ and $s$ are the simplicial vertices of $G_r^s$.
We denote a spinal tree of $r$ in $G_r^s$ by $T_r(G_r^s)$, and so define $T_r=\cup_{s\in S}T_r(G_r^s)$.

For an internal, non-spinal root $r$, with $r$ unique to one fan, let $F_i$ be the fan, centered on $x_i$, that contains $r$.
Then $G$ is the union of two 2-paths $G_r^s$, for $s\in S$, that share the edge $rx_i$.
Note that, for each $s\in S$, we have that $r$ and $s$ are the simplicial vertices of $G_r^s$.
We denote a spinal tree of $r$ in $G_r^s$ by $T_r(G_r^s)$, and so define $T_r=\cup_{s\in S}T_r(G_r^s)$.

Observe that, in the first two cases, $T_r$ is an edge-disjoint union of $\deg(r)-2$ caterpillars, all sharing only the vertex $r$.
In the third case, however, $T_r$ is a union of two caterpillars, each containing the edge $rx_i$.
See Figure \ref{fig:examples} for some examples.
The next result follows from Fact \ref{f:tree}.

\begin{cor}
\label{c:spinal}
Let $G$ be a 2-path of diameter $d$ on $n$ vertices, and $T_r$ be a spinal tree rooted at vertex $r$.
Then $\p(T_r,r) =$
\begin{enumerate}
    \item 
    $2^{\ecc(r)} + 2^{d-\ecc(r)} + n - d - 2$ if $r$ is spinal for some representation, and
    \item 
    $2^{\ecc(r)} + 2^{d+1-\ecc(r)} + n - d - 3$ if $r$ is non-spinal in every representation.
\end{enumerate}
In particular, $\p(T_r,r)\le 2^d+n-d-1$ for all $r$.
\end{cor}

\begin{proof}
Write $S=\{s_1,s_2\}$.
If $r$ is a spinal root then, for each $T_r(G_r^v)$,  $v\in S\cup A_r$, $r$ is a simplicial vertex.
Label the trees so that $T_r(G_r^{s_1})$ has diameter $\ecc(r)$ and $T_r(G_r^{s_2})$ has diameter $d-\ecc(r)$.
The remaining trees are single edges.
Then, by Fact \ref{f:tree}, we have
\begin{align*}
\p(T_r,r)    
&= (2^{\ecc(r)}-1) + (2^{d-\ecc(r)}-1) +(n-d-1)(2^1-1)+1\\
&= 2^{\ecc(r)}+2^{d-\ecc(r)}+n-d-2.
\end{align*}

If $r$ is non-spinal then $r$ is simplicial in each $T_r(G_r^{s_i})$.
Observe that if $r$ is non-spinal in every representation then it is never in three fans; otherwise, if $x_i$ is the center of the middle fan then we obtain a spinal representation for $r$ by using the path $x_{i-1}rx_{i+1}$ in place of $x_{i-1}x_ix_{i+1}$ in the spine.

Suppose it is the case that the edge $rx_i$ is not in the spine of $T_r(G_r^{s_j})$ for some $j$.
We remark in this case that $\diam(T_r) = d + 1$, since we always gain one more edge in the spine of $T_r$ than in the spine of $G$ only.
Label the trees so that $T_r(G_r^{s_1})$ has diameter $\ecc(r)$ and $T_r(G_r^{s_2})$ has diameter $d+1-\ecc(r)$.
Because in each case (see also Figure \ref{fig:othercase}) we have $n-d-2$ vertices not on either spine, Fact \ref{f:tree} implies that
\begin{align*}
\p(T_r,r)
&= (2^{\ecc(r)} -1) +(2^{d+1-\ecc(r)}-1) + (n-d-2)(2^1-1) +1\\
&= 2^{\ecc(r)}+2^{d+1-\ecc(r)}+n-d-3.
\end{align*}

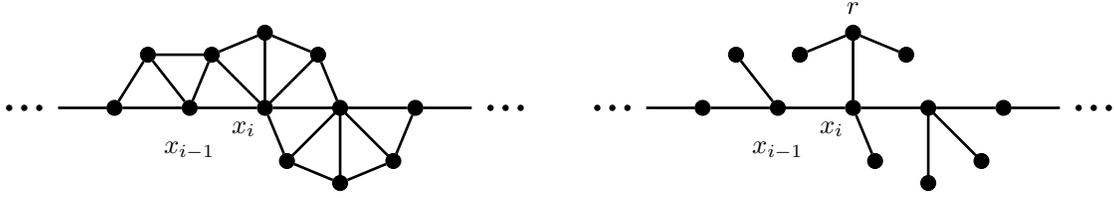
\begin{figure}
\begin{center}
\begin{tikzpicture}[scale=1.0]
\tikzstyle{every node}=[draw,circle,fill=black,inner sep=2pt]
\draw node (r) at (0,0) {};
\draw node (x1) [label=270: {$x_{i-1}$}] at (1,0) {};
\draw node (x2) [label=225: {$x_i$}] at (2,0) {};
\draw node (x3) at (3,0) {};
\draw node (s) at (4,0) {};
\draw[line width=1.0pt] (-.75,0) -- (r) -- (x1) -- (x2) -- (x3) -- (s) -- (4.75,0);
\draw node[inner sep=.7pt] at (-1.0,0) {};
\draw node[inner sep=.7pt] at (-1.2,0) {};
\draw node[inner sep=.7pt] at (-1.4,0) {};
\draw node[inner sep=.7pt] at (5.0,0) {};
\draw node[inner sep=.7pt] at (5.2,0) {};
\draw node[inner sep=.7pt] at (5.4,0) {};
\def \j {1}
\def \A {180/4}
\def \B {11*180/16}
    \draw node (u) at ({1+cos(\B)},{sin(\A)}) {};
    \draw[line width=1.0pt] (r) -- (u) -- ({2+cos(180-\A)},{sin(180-\A)});
    \draw[line width=1.0pt] (x1) -- (u);
\def \j {2}
\def \a {3}
\def \A {180/(1+\a)}
    \draw node (v1) at ({\j+cos(180-1*\A)},{sin(180-1*\A)}) {};
    \draw node (v2) at ({\j+cos(180-2*\A)},{sin(180-2*\A)}) {};
    \draw node (v3) at ({\j+cos(180-3*\A)},{sin(180-3*\A)}) {};
    \draw[line width=1.0pt] (x1) -- (v1) -- (v2) -- (v3) -- (x3);
    \draw[line width=1.0pt] ({\j+cos(180-1*\A)},{sin(180-1*\A)}) -- (\j,0);
    \draw[line width=1.0pt] ({\j+cos(180-2*\A)},{sin(180-2*\A)}) -- (\j,0);
    \draw[line width=1.0pt] ({\j+cos(180-3*\A)},{sin(180-3*\A)}) -- (\j,0);
\def \j {3}
\def \a {3}
\def \A {180/(1+\a)}
    \draw node (w1) at ({\j+cos(180+1*\A)},{sin(180+1*\A)}) {};
    \draw node (w2) at ({\j+cos(180+2*\A)},{sin(180+2*\A)}) {};
    \draw node (w3) at ({\j+cos(180+3*\A)},{sin(180+3*\A)}) {};
    \draw[line width=1.0pt] (x2) -- (w1) -- (w2) -- (w3) -- (s);
    \draw[line width=1.0pt] ({\j+cos(180+1*\A)},{sin(180+1*\A)}) -- (\j,0);
    \draw[line width=1.0pt] ({\j+cos(180+2*\A)},{sin(180+2*\A)}) -- (\j,0);
    \draw[line width=1.0pt] ({\j+cos(180+3*\A)},{sin(180+3*\A)}) -- (\j,0);
\end{tikzpicture}
$\qquad$
\begin{tikzpicture}[scale=1.0]
\tikzstyle{every node}=[draw,circle,fill=black,inner sep=2pt]
\draw node (r) at (0,0) {};
\draw node (x1) [label=270: {$x_{i-1}$}] at (1,0) {};
\draw node (x2) [label=225: {$x_i$}] at (2,0) {};
\draw node (x3) at (3,0) {};
\draw node (s) at (4,0) {};
\draw[line width=1.0pt] (-.75,0) -- (r) -- (x1) -- (x2) -- (x3) -- (s) -- (4.75,0);
\draw node[inner sep=.7pt] at (-1.0,0) {};
\draw node[inner sep=.7pt] at (-1.2,0) {};
\draw node[inner sep=.7pt] at (-1.4,0) {};
\draw node[inner sep=.7pt] at (5.0,0) {};
\draw node[inner sep=.7pt] at (5.2,0) {};
\draw node[inner sep=.7pt] at (5.4,0) {};
\def \j {1}
\def \A {180/4}
\def \B {11*180/16}
    \draw node (u) at ({1+cos(\B)},{sin(\A)}) {};
    \draw[line width=1.0pt] (x1) -- (u);
\def \j {2}
\def \a {3}
\def \A {180/(1+\a)}
    \draw node (v1) at ({\j+cos(180-1*\A)},{sin(180-1*\A)}) {};
    \draw node (v2) [label=above: {$r$}] at ({\j+cos(180-2*\A)},{sin(180-2*\A)}) {};
    \draw node (v3) at ({\j+cos(180-3*\A)},{sin(180-3*\A)}) {};
    \draw[line width=1.0pt] (v1) -- (v2) -- (v3);
    \draw[line width=1.0pt] ({\j+cos(180-2*\A)},{sin(180-2*\A)}) -- (\j,0);
\def \j {3}
\def \a {3}
\def \A {180/(1+\a)}
    \draw node (w1) at ({\j+cos(180+1*\A)},{sin(180+1*\A)}) {};
    \draw node (w2) at ({\j+cos(180+2*\A)},{sin(180+2*\A)}) {};
    \draw node (w3) at ({\j+cos(180+3*\A)},{sin(180+3*\A)}) {};
    \draw[line width=1.0pt] (x2) -- (w1);
    \draw[line width=1.0pt] ({\j+cos(180+2*\A)},{sin(180+2*\A)}) -- (\j,0);
    \draw[line width=1.0pt] ({\j+cos(180+3*\A)},{sin(180+3*\A)}) -- (\j,0);
\end{tikzpicture}
\caption{A 2-path with $rx_i$ in the spines of both trees $T_r(G_r^{s_1})$ and $T_r(G_r^{s_2})$.}
\label{fig:othercase}
\end{center}
\end{figure}

All these formulas are maximized when $\ecc(r)$ is as large as possible.
Hence we have
\begin{align*}
    \max\{2^{d-1}+2+n-d-2,2^d+2+n-d-3\}
    &= \max\{2^d+n-d-2^{d-1},2^d+n-d-1\}\\
    &= 2^d+n-d-1,
\end{align*}
as claimed.
\end{proof}

\begin{lem}
\label{l:SimpTree}
If $r^*$ is a simplicial target of a 2-path $G$, with rooted spinal tree $T^* = T_{r^*}$, then $\p(T_r,r)\le\p(T^*,r^*)$.
\end{lem}

\begin{proof}
Corollary \ref{c:spinal} shows that the pebbling numbers for each type of $r$ in $G$ are bounded above by $2^d+n-d-1$, which equals the pebbling number at $r^*$ by Fact \ref{f:tree}.
\end{proof}

\begin{lem}
\label{l:Cheap}
Let $G$ be a 2-path and $r$ be any vertex.
Then $\p_2(G)\ge \p(T^*,r^*)$.
Thus, any configuration of size at least $\p_2(G)$ is $r$-cheap.
\end{lem}

\begin{proof}
Note that for any $d$, $2^d > d$.
Then by Theorem \ref{t:2Path} we have:
\begin{align*}
\p(T^*,r^*)
&=2^d+n-d-1\\
&\le 2^d + n - d - 1 + (2^d - d + 1)\\
&= 2^{d+1} + n - 2d\\
&= \p_2(G).
\end{align*}
The existence of a cheap solution follows from the Cheap Lemma \ref{l:CheapLemma}.
\end{proof}

\subsection{Verification of the Target Conjecture}

\begin{thm}
\label{t:2PathTarget}
Let $G$ be a 2-path and $D$ be a target distribution of size $t$.
Then $\p_{t}(G) \ge \p(G,D)$.
\end{thm}

\begin{proof}

Let $|C| = \p_{t}(G)$.
We proceed by induction.
If $t=1$ then $\p_{1}(G) \ge \p(G,r)$ for any $r$, so $C$ solves any $D$ of size 1.

Now let $t \ge 2$ and recall from Theorem \ref{t:2Path} that $\p_{t}(G) = \p_{t-1}(G) + 2^d$.
By the induction hypothesis, we assume that $\p(G,D') \le \p_{t-1}(G)$ for all $|D'| = t-1$.
Let $D$ be given.
Choose any $r\in\dD$ and let $T_r$ denote a spinal tree of $r$. 
Then, since $|C|\ge\p_2(G)\ge\p(T_r,r)$, we know by the Cheap Lemma \ref{l:Cheap} that $C$ has a cheap $r$-solution $\s$.
Thus $|C-\s| = |C| - \cost(\s) \ge \p_{t}(G,D) - 2^d = \p_{t-1}(G,D-r)$.
Therefore $C-\s$ solves $D-r$ by induction, and hence $C$ solves $D$.
\end{proof}

\section{Kneser graphs}
\label{s:Kneser}

\subsection{Definition}

For $m\ge h\ge 1$, the {\it Kneser graph} $K(m,h)$ is the graph whose vertex set equals the set all subsets of $[m]=\{1,2,\ldots,m\}$ of size $h$ --- so that $|V(K(m,h))| = \binom{m}{h}$ --- with two vertices adjacent if and only if they are disjoint.
Observe that $K(m,1)$ is the complete graph on $m$ vertices and that $K(5,2)$ is the Petersen graph.
Also note that $K(m,h)$ is empty for $m<2h$, is a matching for $m=2h$, is connected for $m\ge 2h+1$, has $\diam=2$ for $m\ge 3h-1$, and is vertex transitive.
The following known theorems are used in the proofs below.
Let $\k(G)$ denote the connectivity of a graph $G$.

\begin{thm}
\label{t:KneserReg}
\cite{CheLih}
For $m\ge 2h+1\ge 3$, $K(m,h)$ is connected, edge transitive, and regular of degree $\binom{m-h}{h}$.
\end{thm}

\begin{thm}
\label{t:RegularConn}
(Exercise 15(c) of Chapter 12 of \cite{Lovas})
If a graph $G$ is connected, edge transitive, and regular of degree $\d$, then $\k(G)=\d$.
\end{thm}

\begin{cor}
\label{c:KneserConn}
For $m\ge 2h+1\ge 3$, $\k(K(m,h))=\binom{m-h}{h}$.
\end{cor}

\begin{proof}
Theorem \ref{t:KneserReg} shows that $K(m,h)$ satisfies the hypothesis of Theorem \ref{t:RegularConn}, from which the result follows.
\end{proof}

We will use the following general form of Menger's Theorem (see Exercise 4.2.28 of \cite{West}).

\begin{thm}
\label{t:MengerGen}
Let $X$ and $Y$ be disjoint sets of vertices in a $k$-connected graph $G$.
Let $U(x)$, for $x\in X$, and $W(y)$, for $y\in Y$, be nonnegative integers such that $\sum_{x\in X}U(x) = \sum_{y\in Y}W(y) = k$.
Then $G$ has $k$ pairwise internally disjoint $X,Y$-paths so that
$U(x)$ of them start at $x$ and $W(y)$ of them end at $y$ for all $x\in X$ and $y\in Y$.
\end{thm}

The following pebbling result is also known.

\begin{thm}
\label{t:KneserClass0}
\cite{Hurl}
For all $m \ge 5$ we have $\p(K(m,2)) = \binom{m}{2}$.
\end{thm}

\subsection{Results}

Let $n=|V(K(m,2))|=\binom{m}{2}$.
For any vertex $r$, we define the configuration $C_{t,1}$ to have $2t-1$ pebbles on one $x\in V_2(r)$, and one pebble on every other vertex except for $r$. 
Note that $|C_{t,1}| = 2t - 1 + n - 2 = n + 2t - 3$.
In addition, we define the configuration $C_{t,2}$ to have $4t-1$ pebbles on one $y\in V_2(r)$, and one pebble on every other $v\in V_2(r)$.
Observe that $|C_{t,2}| = 4t - 1 + 2(m-2) - 1 = 4t + 2(m-2) - 2$.
We prove in Lemma \ref{l:KneserLowerBound} that each $C_{t,i}$ is $r$-unsolvable.

Define $p_1(m,t) = |C_{t,1}|+1 = n+2t-2$, $p_2(m,t) = |C_{t,2}|+1 = 4t+2m-5$, $p(m,t) = \max_i p_i(m,t)$, and $t_0 = \deg(K(m,2))/2 = \binom{m-2}{2}/2$.
Then 
$$p(m,t)=\Bigg\{
\begin{tabular}{ll}
    $p_1(m,t)$&{\rm for\ }$t\le t_0$,\\
    $p_2(m,t)$&{\rm for\ }$t\ge t_0$.\\
\end{tabular}
$$

\begin{lem}
\label{l:KneserLowerBound}
Let $G=K(m,2)$ for any $m \ge 5$.
Then $\p_t(G) \ge p(m,t)$.
\end{lem}

\begin{proof}
Because $\pot(C_{t,1})=t-1$, the Potential Lemma \ref{l:potlem}(\ref{potlem2}) implies that $C_{t,1}$ is not $t$-fold $r$-solvable for some root $r$.
Furthermore, we have $\pot(C_{t,2}) = 2t - 1$, with the only potential vertex $u$ at distance two from $r$.
Thus $C_{t,2}$ can place at most $2t-1$ pebbles on $V_1(r)$, resulting in a configuration $C'$ with potential at most $t-1$.
By the Potential Lemma \ref{l:potlem}(\ref{potlem2}), $C'$, and hence $C$, is not $t$-fold $r$-solvable.
\end{proof}

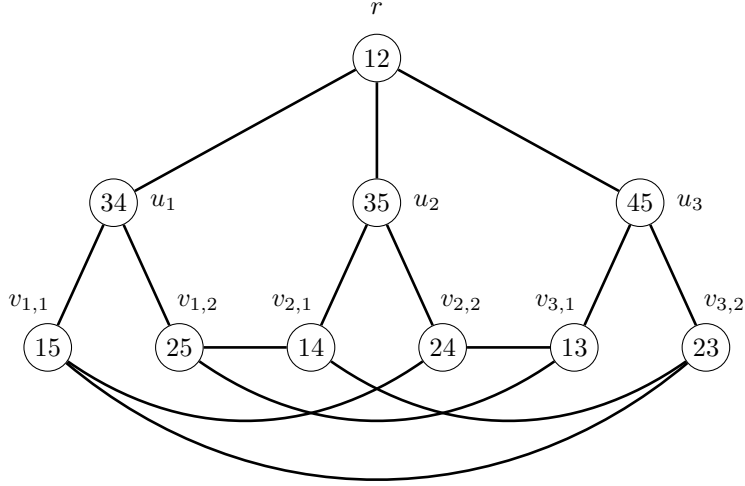
\begin{figure}
\begin{center}
\begin{tikzpicture}[scale=1.75]
\tikzstyle{every node}=[draw,circle,minimum size=18pt,inner sep=0.2pt,label distance=0.02cm]
\def \h {1.1}
\def \d {3}
\def \a {35}
\def \b {40}
\draw node (12) [label=above: {$r$}] at (2.5,2*\h) {12};
\draw node (34) [label=right: {$u_1$}] at (0.5,\h) {34};
\draw node (35) [label=right: {$u_2$}] at (2.5,\h) {35};
\draw node (45) [label=right: {$u_3$}] at (4.5,\h) {45};
\draw node (15) [label=90+\d: {$v_{1,1}$}] at (0,0) {15};
\draw node (25) [label=90-\d: {$v_{1,2}$}] at (1,0) {25};
\draw node (14) [label=90+\d: {$v_{2,1}$}] at (2,0) {14};
\draw node (24) [label=90-\d: {$v_{2,2}$}] at (3,0) {24};
\draw node (13) [label=90+\d: {$v_{3,1}$}] at (4,0) {13};
\draw node (23) [label=90-\d: {$v_{3,2}$}] at (5,0) {23};
\draw[line width=1.0pt] (12) -- (34);
\draw[line width=1.0pt] (12) -- (35);
\draw[line width=1.0pt] (12) -- (45);
\draw[line width=1.0pt] (34) -- (15);
\draw[line width=1.0pt] (34) -- (25);
\draw[line width=1.0pt] (35) -- (14);
\draw[line width=1.0pt] (35) -- (24);
\draw[line width=1.0pt] (45) -- (13);
\draw[line width=1.0pt] (45) -- (23);
\draw[line width=1.0pt] (25) -- (14);
\draw[line width=1.0pt] (24) -- (13);
\draw[line width=1.0pt] (15) to[out=-\a,in=180+\a] (24);
\draw[line width=1.0pt] (25) to[out=-\a,in=180+\a] (13);
\draw[line width=1.0pt] (14) to[out=-\a,in=180+\a] (23);
\draw[line width=1.0pt] (15) to[out=-\b,in=180+\b] (23);
\end{tikzpicture}\\
\caption{The Petersen graph $P=K(5,2)$, labeled as in the proof of Lemma \ref{l:PetersenUpperBound}.}
\label{fig:petersen}
\end{center}
\end{figure}

\begin{lem}
\label{l:PetersenUpperBound}
Let $G=K(5,2)$ and $|D| = t$.
Then $\p(G, D) \le p(5,t)$.
\end{lem}

\begin{proof}
Let $r$ be any target in $D$.
Label the vertices as follows (see Figure \ref{fig:petersen}):
$V_1(r) = \{u_i\mid i\in [3]\}$ and $V_2(r) = \{v_{i,j}\mid i\in [3], j\in [2]\}$, with adjacencies $u_i\sim v_{i,j}$ for all $i$ and $j$ and $v_{i,j}\sim v_{i',j'}$ if and only if $i\not= i'$ and $j\not= j'$.
Define the configuration $J_r$ by $J_r(r)=0$ and $J_r(x)=1$ otherwise.
Also, define $U_i=\{u_i,v_{i,1},v_{i,2}\}$ and write $J_{r,i}$ for the restriction of $J_r$ to $U_i$.
For a configuration $C$, we say that a vertex $v$ is {\it big} if $C(v)\ge 2$.
\\
{\bf Claim A.} 
If $C$ is $r$-unsolvable of size 9 then $C=J_r$.
\\
{\it Proof.}
Let $C$ be $r$-unsolvable of size 9.
Suppose that $\pot(C)>0$.
Since $C$ does not solve $r$, no neighbor of $r$ is big.
If $C(U_i)\ge 5$ then $C$ solves $r$, either by having two pebbles on $u_i$, one pebble on $u_i$ and moving a pebble to $u_i$, or no pebbles on $u_i$ and moving two pebbles to $u_i$.
Thus we may assume that $C(U_i)\le 4$ for all $i$.

If some $C(U_i)=4$ then some $C(U_{i'})\ge 3$.
If $U_{i'}$ has a big vertex, it can move a fifth pebble into $U_i$, which can then solve $r$.
Thus $C(U_{i'})=J_{r,{i'}}$.
However, this allows a big vertex from $U_i$ to slide through $U_{i'}$ to $r$, which is a contradiction.

Hence $C(U_i)=3$ for all $i$.
If $C$ has two potentials then two pebbles can be moved into a third $U_i$, which yields 5 pebbles that can solve $r$.
Thus we may assume that $C$ has exactly one potential, say in $U_i$.
But then some $i'$ has $C(U_{i'})=J_{r,{i'}}$, and so the big vertex can slide through $U_{i'}$ to $r$, a contradiction.

Hence $\pot(C)=0$; i.e. $C=J_r$.
\hfill $\diamondsuit$
\\
{\bf Claim B.}
If $|C|\ge 13$ then $C$ is $r$-cheap.
\\
{\it Proof.}
By the pigeonhole principle we have that some $C(U_i)\ge 5$.
Thus $U_i$ has a big vertex $w$.
We have a cost-2 solution if $w=u_i$, so consider that $w=v_{i,j}$ for some $j$.
We have a cost-3 solution if $u_i$ has a pebble, so consider that $u_i$ is empty.
Then $U_i$ has two potentials, which produces a cost-4 solution.
\hfill $\diamondsuit$

Because $t_0=\binom{3}{2}/2=1.5$, we have $p(5,t)=p_1(5,t)$ for $t=1$ and $p(5,t)=p_2(5,t)$ for $t\ge 2$.
The statement is true for $t=1$ because we know from Theorem \ref{t:KneserClass0} that $\p(K(5,2))=10=p_1(5,1)$.

For $|D|=t=2$ we have $|C|=p_2(5,2)=13$.
Let $r_1$ and $r_2$ be the two (not necessarily distinct) vertices with $D(r_i)>0$.
Claim B implies that $C$ has a cheap $r_1$-solution $\s$.
If $C-\s$ has an $r_2$-solution then we're done, so assume otherwise.
It must be then that $|C-\s|=9$, and so Claim A implies that $C-\s=J_{r_2}$.
Thus $r_2$ is the only empty vertex in $C$ (it is the only empty vertex in $C-\s$, and if it is not empty in $C$ then $C$ solves $D$ without moving), so any big vertex has a greedy slide; i.e. an $r_2$-solution $\s'$ of cost at most 3.
Then $C-\s'$ has size at least 10, which provides an $r_1$-solution.
Thus $C$ solves $D$.

For $|D|=t\ge 3$ we have $|C|=p_2(5,t)=4t+5$.
Fix a target $r$ with $D(r) > 0$.
Then Claim B implies that $C$ has a cheap $r$-solution $\s$, and so $|C-\s|\ge |C|-4=4(t-1)+5=p_2(5,t-1)$.
By induction, $C- \s$ is $(D-r)$-solvable.
Hence $C$ is $D$-solvable.
\end{proof}

\begin{fct}
\label{f:CommonNeighbors}
Let $G=K(m,2)$ for some $m > 5$ and let $r$ be any vertex.
Then for every (not necessarily distinct) $u, v\in V_2(r)$ there exists a $w\in V_1(r)$ such that $w\in N(u)\cap N(v)$.
\end{fct}

\begin{proof}
For every $x\in V_2(r)$, we have $|x\cap r|=1$.
Thus, for every $u,v\in V_2(r)$, we have $|(r\cup u\cup v)|\le 4$.
Because $m\ge 6$, there exists $y=\{c,d\}\subseteq [m]-(r\cup u\cup v)$.
Thus $y\in N(r)\cap N(u)\cap N(v)$.
\end{proof}

\begin{lem}
\label{l:KneserUpperBound}
Let $G = K(m,2)$ for some  $m>5$ and let $|D| = t$.
Then $\p(G, D) \le p(m,t)$.
\end{lem}

\begin{proof}
We prove the upper bound using induction on $t$.
Because of Theorem \ref{t:KneserClass0}, the statement is true for $t=1$, so we assume that $t\ge 2$.
Let $C$ be a configuration on $G$ of size $p(m,t)$.
If some vertex $r$ has $C(r)>0$ and $D(r)>0$, then a pebble on $r$ solves a target with cost 1.
Since $|C-r|=|C|-1>p(m,t-1)$, we know by induction that $C-r$ solves $D-r$.
Hence we may assume that no pebbles of $C$ already sit on any target of $D$.

From Fact \ref{f:CommonNeighbors}, we see that, for any $r$, every pair of potentials yields an $r$-solution when $m>5$.
Indeed, if one of the potentials is in $V_1(r)$ then it solves $r$ immediately.
Otherwise both are in $V_2(r)$ and have a common neighbor to move two pebbles to, solving $r$ subsequently.
Notice that both of these solutions are greedy, and hence cheap.

Define $\dP$ to be the multiset of vertices on which the $\pot(C)$ potentials of $C$ sit, and $Z$ to be the set of $z$ zeroes.
Note that $\pot(C)\ge t$ by Lemma \ref{l:potlem}(\ref{potlem1}).
Set $t'=\min\{t,2t_0\}$, and let $\dP'$ be any submultiset of $\dP$ of size $t'$, and $\dD'$ be any submultiset of $\dD$ of size $t'$.
By Corollary \ref{c:KneserConn} and Theorem \ref{t:MengerGen}, since $\k(G)=2t_0$, $G$ has $t'$ internally disjoint paths $\cP$ from $\dP'$ to $\dD'$.

Now define $s$ to be the number of paths of $\cP$ that are slides.
If $s=t$ then of course we are done, so we assume that $s<t$.
Then $(t' - s)$ of the paths of $\cP$ are not slides, and therefore have zeroes on them.
Let $z = |Z|$ and note that $z \ge 2t_0 - s + 1$.
Indeed, let $Z'=Z-\dD'$, $G'=G-Z$, and suppose that $|Z'| \le 2t_0 - s - 1$.
Then $\k(G') \ge \k(G) - |Z'| \ge 2t_0 - (2t_0 - s - 1) = s+1$.
Thus there are at least $s + 1$ paths from $\dP'$ to $\dD'$ --- these are all slides of $C$ in $G$.
Hence, if there are exactly $s$ slides of $C$ in $G$, then $|Z'| \ge 2t_0 - s$.
Therefore $|Z| = |Z'| + s(D') \ge (2t_0 - s) + 1 = 2t_0 - s + 1$.
We will show that the remaining potential after using $s$ slides is large enough to solve the remaining $t - s$ solutions.

We first consider the case when $t \le t_0$, which implies that $t'=t$.
Furthermore, it forces $2\le t_0=\binom{m-2}{2}/2$, which requires $m>5$.
In this case we have $|C| = p_1(m,t) = n+2t-2$.
Then $s = \lfloor (s+1)/2\rfloor + \lfloor s/2\rfloor \ge \lfloor (s+1)/2\rfloor$, and so
\begin{align*}
    \pot(C) - s 
    &\ge \Big\lceil \big[(n+2t-2) - n + (2t_0 - s + 1)\big] / 2  \Big\rceil - s \\
    &= \big\lceil (2t + 2t_0 - s - 1) / 2  \big\rceil - s \\
    &\ge t + t_0 - s - \big\lfloor (s + 1)/ 2 \big\rfloor \\
    &\ge 2t - s - s\\
    &\ge  2(t - s)\ .
\end{align*} 
Thus, after solving $s$ slides, we have at least $2(t - s)$ potential.
As noted above, for each remaining target, we can solve it from any pair of potentials, consequently solving all $t - s$ remaining targets from the at least $2(t-s)$ remaining potentials.

Second, we consider the case when $t_0 < t = \lceil t_0\rceil$ (still we have $t'=t$).
In this case we have $|C| = p_2(m,t) = 4t+2m-5$.
Then because $n = |\{r\}| + |V_1(r)| + |V_2(r)| = 1 + \binom{m-2}{2} + 2(m-2)$ for any $r$, we have
\begin{align*}
    \pot(C) - s 
    &\ge \Big\lceil \big[(4t+2m-5) - n + (2t_0 - s + 1)\big] / 2  \Big\rceil - s \\
    &= \Big\lceil \big[4t -\big(n-1-2t_0-2(m-2)\big) -(s+1)\big] / 2  \Big\rceil - s \\
    &= 2t - \lfloor (s+1) / 2  \rfloor - s\\
    &\ge 2t - s - s\\
    &\ge  2(t - s)\ .
\end{align*} 
Thus, after solving $s$ slides, we have at least $2(t - s)$ potential, which gives us the $t - s$ remaining solutions, as before.
        
Finally, we consider the case when $t > \lceil t_0\rceil$.
In this case we have $|C| = p_2(m,t) = 4t + 2m - 5$.
Then 
\begin{align*}
    \pot(C) 
    &\ge \Big\lceil \big([4t + 2m - 5] - n + 1\big) / 2 \Big\rceil \\
    &= \Big\lceil \left(4t - \big[n - 2(m-2)\big]\right) / 2 \Big\rceil\\
    &\ge \left\lceil \frac{1}{2}\left(4t - \left[\binom{m-2}{2} + 1\right]\right) \right\rceil\\
    &\ge \left\lceil \frac{1}{2}\left(2t + 2t - \binom{m-2}{2} - 1\right) \right\rceil\\
    &\ge \left\lceil \frac{1}{2}\left(2t - 1 + \left[2t - \binom{m-2}{2}\right]\right) \right\rceil\\
    &\ge \lceil (t - 1/ 2) + (t-t_0) \rceil\\
    &\ge \lceil (t + 1/ 2) \rceil\\
    &> t\\
    &\ge 2.
\end{align*}
As noted above, $\pot(C) \ge 2$ yields a cheap $r$-solution $\s$ for any $r\in\dD$. 
After using $\s$, the remaining configuration $C'$ has size at least $|C|-4=p_2(m,t-1)$.
Thus, we can use induction on $t$ to get from $C'$ the $t-1$ remaining solutions of $D-r$. 
\end{proof}

\subsection{Verification of the Target Conjecture}

\begin{cor}
\label{c:KneserFormula}
Let $G= K(m,2)$ for $m \ge 5$.
Then $\pi_t(G) = p(m,t)$.
\end{cor}

\begin{proof}
For $m=5$ this follows from Lemmas \ref{l:KneserLowerBound} and \ref{l:PetersenUpperBound}.
For $m>5$ this follows from Lemmas \ref{l:KneserLowerBound} and \ref{l:KneserUpperBound}.
\end{proof}

\begin{thm}
\label{t:Kneser}
Let $G=K(m,2)$ and $|D|=t$.
Then for all $D$, $\p(G,D)\le \p_t(G)$.
\end{thm}

\begin{proof}
This follows from Lemmas \ref{l:PetersenUpperBound} and \ref{l:KneserUpperBound} and from Corollary \ref{c:KneserFormula}.
\end{proof}

\section{Remarks}

A natural next step for verifying the Target Conjecture would be to consider $k$-paths for $k\ge 3$.
However, the $t$-fold pebbling numbers for this family are not presently known.
In \cite{AlcoHurl} the subfamily of $k^{\rm th}$ powers of paths is studied.
Additional interesting families to investigate include diameter two graphs, Class 0 graphs, and chordal graphs, among others.
Unfortunately, $t$-fold pebbling numbers are not known for such broad classes, so more specific subclasses need to be examined.

\end{document}